\documentclass[11pt]{article}
\usepackage{amssymb}
\usepackage{amsmath}
\usepackage{amsthm}
\usepackage{array}
\usepackage{cite}

\theoremstyle{plain}

\newtheorem{theorem}{Theorem}
\newtheorem{proposition}{Proposition}

\newtheorem{corollary}{Corollary}

\theoremstyle{remark}

\renewcommand{\qed}{\hfill\rule{1ex}{1ex}}
\renewenvironment{proof}[1]{\trivlist\item[\hskip\labelsep{\bf Proof{#1}}.]}{\qed\endtrivlist}

\begin{document}

\title{Semiring identities of semigroups of reflexive relations and upper triangular boolean matrices\thanks{Supported by the Russian Science Foundation (grant No. 22-21-00650).}}
\author{S. V. Gusev}
\date{\small Institute of Natural Sciences and Mathematics\\
Ural Federal University\\
620000 Ekaterinburg, RUSSIA\\
\texttt{sergey.gusb@gmail.com}}
\maketitle

\begin{abstract}
We show that the following semirings satisfy the same identities: the semiring $\mathcal{R}_n$ of all reflexive binary relations on a set with $n$ elements, the semiring $\mathcal{U}_n$ of all $n\times n$ upper triangular matrices over the boolean semiring, the semiring $\mathcal{C}_n$ of all order preserving and extensive transformations of a chain with $n$ elements. In view of the result of Kl\'ima and Pol\'ak, which states that $\mathcal{C}_n$ has a finite basis of identities for all $n$, this implies that the identities of $\mathcal{R}_n$ and $\mathcal{U}_n$ admit a finite basis as well.
\end{abstract}

An \emph{additively idempotent semiring} (ai-semiring, for short) is an algebra $\mathcal{S}=(S, +, \cdot)$ of type $(2, 2)$ such that the additive reduct $(S,+)$ is a \emph{semilattice} (that is, a commutative idempotent semigroup), the multiplicative reduct $(S,\cdot)$ is a semigroup, and multiplication distributes over addition on the left and on the right, that is, $\mathcal{S}$ satisfies the identities $x(y+z)\approx xy+xz$ and $(y+z)x\approx yx+zx$.

The set of all reflexive binary relations on a set with $n$ elements forms an ai-semiring under set-theoretical union and multiplication. This ai-semiring can be conveniently thought of as a subsemiring of the ai-semiring of all $n\times n$ matrices (with the usual matrix multiplication and addition) over the boolean semiring $\mathcal{B}=\langle\{0,1\};+,\cdot\rangle$ with the operations defined by the rules 
\[
0\cdot 0=0\cdot 1=1\cdot 0=0+0=0,\quad 1\cdot 1=1+0=0+1=1+1=1.
\]
Namely, it can be identified with the subsemiring consisting of matrices in which all diagonal entries are~1. Denote this semiring by $\mathcal{R}_n=\langle R_n;+,\cdot\rangle$. By $\mathcal{U}_n=\langle U_n;+,\cdot\rangle$ we denote the subsemiring of $\mathcal{R}_n$ consisting of upper triangular matrices, that is, matrices $\bigl(\alpha_{ij}\bigr)$ with $\alpha_{ij}=0$ for $j<i$. 

A transformation $\alpha$ of a partially ordered set $\langle Q;\le\rangle$ is called \emph{order preserving} if $q\le q'$ implies $q.\alpha\le q'.\alpha$ for all $q,q'\in Q$, and \emph{extensive} if $q\le q.\alpha$ for every $q\in Q$. We say that an ai-semiring $\langle S;+,\cdot\rangle$ is a \emph{join-semiring} of order preserving and extensive transformations of a join-semilattice $\langle Q;\le\rangle$ if $\langle S;\cdot\rangle$ is a semigroup of order preserving and extensive transformations of $\langle Q;\le\rangle$ and $q.(\alpha+\beta)=\mathsf{sup}(q.\alpha,q.\beta)$ for all $\alpha,\beta\in S$ and $q\in Q$.  For example, the set of all order preserving and extensive transformations of a chain with $n$ elements forms a join-semiring. Denote this semiring by $\mathcal{C}_n=\langle C_n;+,\cdot\rangle$.

In~\cite{Volkov-04}, Volkov shows that the monoids $\langle R_n;\cdot\rangle$, $\langle U_n;\cdot\rangle$ and $\langle C_n;\cdot\rangle$ satisfy the same identities and these identities admit a finite basis if and only if $n\le 4$. In contrast, by the result of Kl\'ima and Pol\'ak~\cite{Klima-Polak-10}, the ai-semiring $\mathcal{C}_n$ has a finite basis of identities for each $n$. In the present note we prove the following theorem.

\begin{theorem}
\label{finite basis}
For each $n$, the three ai-semirings $\mathcal{U}_n$, $\mathcal{R}_n$ and
$\mathcal{C}_n$ satisfy the same identities and these identities admit a finite basis.
\end{theorem}

To prove Theorem~\ref{finite basis}, we need some definitions, notation and auxiliary results. If $u,v$ are words over the same alphabet $\Sigma$, we say that $u$ is a \emph{subword} of $v$ whenever there exist words $u_1,\dots,u_n,v_0,v_1,\dots,v_{n-1},v_n\in\Sigma^\ast$ such that
\[
u=u_1\cdots u_n\quad \text{and}\quad v=v_0u_1v_1\cdots v_{n-1}u_nv_n;
\]
in other terms, this means that one can extract $u$ treated as a sequence of letters from the sequence $v$. Let $\mathsf s_k(w)$ denote the set of all subwords of $w$ of length $\le k$. Recall that a \emph{semiring identity} over an alphabet $\Sigma$, or simply \emph{identity}, is merely a pair $(u_1+\cdots+u_\ell,v_1+\cdots+v_r)$, where $u_1,\dots,u_\ell,v_0,\dots,v_r\in\Sigma^+$, usually written as 
\begin{equation}
\label{identity}
u_1+\cdots+u_\ell\approx v_1+\cdots+v_r.
\end{equation}
We denote by $J_k$ the set of all identities~\eqref{identity} with $\bigcup_{i=1}^\ell \mathsf s_k(u_i)=\bigcup_{i=1}^r \mathsf s_k(v_i)$. For an ai-semiring $\mathcal{S}$, we denote by $\mathsf{Id}(\mathcal{S})$ the set of all identities of $\mathcal{S}$. 

\begin{proposition}
\label{order}
Let $\mathcal{S}=\langle S;+,\cdot\rangle$ be a join-semiring of order preserving and extensive transformations of a join-semilattice $\langle Q;\le\rangle$. If $k+1$ is the length of the longest chain in $\langle Q;\le\rangle$, then $\mathcal{S}$ satisfies every identity in  $J_k$.
\end{proposition}

\begin{proof}{}
Take any identity~\eqref{identity} in $J_k$ and let $\Sigma$ be the alphabet of the words $u_1,\dots,u_\ell$ and $v_1,\dots,v_r$. We have to show that for every substitution $\varphi\colon\Sigma\to S$ one gets $(u_1+\cdots+u_\ell)\varphi=(v_1+\cdots+v_r)\varphi$ or, equivalently, $q.(u_1+\cdots+u_\ell)\varphi=q.(v_1+\cdots+v_r)\varphi$ for all $q\in Q$.

Thus, fix an arbitrary substitution $\varphi\colon\Sigma\to S$ and an arbitrary element $q_0\in Q$. By symmetry, it suffices to verify that 
\[
q_0.(u_1+\cdots+u_\ell)\varphi\le q_0.(v_1+\cdots+v_r)\varphi.
\] 
If $q_0.u_i\varphi=q_0$ for all $i=1,\dots,\ell$, then 
\[
q_0.(u_1+\cdots+u_\ell)\varphi=q_0\le q_0.(v_1+\cdots+v_r)\varphi
\]
because the transformation $(v_1+\cdots+v_r)\varphi$ is extensive. Suppose now that the set $\{i_1,\dots,i_p\}=\{i\mid 1\le i\le \ell,\ q_0.u_i\varphi>q_0\}$ is not empty. For any $i=i_1,\dots,i_p$, denote by $u_{i1}$ the longest prefix of the word $u_i$ such that $q_0.u_{i1}\varphi=q_0$ and let $x_{i1}\in\Sigma$ be the letter that follows $u_{i1}$ in $u_i$ so that $u_i=u_{i1}x_{i1}w_{i1}$ for some $w_{i1}\in\Sigma^\ast$. Then
\begin{equation}
\label{step1}
q_{i1}=q_0.(u_{i1}x_{i1})\varphi=q_0.u_{i1}\varphi x_{i1}\varphi=q_0.x_{i1}\varphi\ge q_0
\end{equation}
because the transformation $x_{i1}\varphi$ is extensive, and by the choice of the prefix $u_{i1}$ the inequality $q_1\ge q_0$ is in fact strict. Now denote by $u_{i2}$ the longest prefix of the word $w_{i1}$ such that $q_{i1}.u_{i2}\varphi=q_{i1}$ and let $x_{i2}\in\Sigma$ be the letter that follows $u_{i2}$ in $w_{i1}$ so that $u_i=u_{i1}x_{i1}u_{i2}x_{i2}w_{i2}$ for some $w_{i2}\in\Sigma^\ast$. Then
\begin{equation}
\label{step2}
q_{i2}=q_{i1}.(u_{i2}x_{i2})\varphi=q_{i1}.u_{i2}\varphi x_{i2}\varphi=q_{i1}.x_{i2}\varphi>q_{i1}
\end{equation}
and substituting the expressions for $q_{i1}$ from~\eqref{step1} in the expressions for $q_{i2}$ in~\eqref{step2}, we also get
\[
q_{i2}=q_0.(u_{i1}x_{i1}u_{i2}x_{i2})\varphi=q_0.(x_{i1}x_{i2})\varphi.
\]
Continuing this process, we finally arrive at the decomposition
\begin{equation}
u_i=u_{i1}x_{i1}u_{i2}x_{i2}\cdots x_{im_i}u_{m_i+1}
\label{decomposition}
\end{equation}
such that $q_0.u_i\varphi=q_0.(x_{i1}\cdots x_{im_i})\varphi$ and
\[
q_{im_i}>q_{i,m_i-1}>\dots>q_{i1}>q_0
\]
where $q_{ij}=q_0.(x_{i1}\cdots x_{ij})\varphi$ for $j=1,2,\dots,m_i$. Since the longest chain in $\langle Q;\le\rangle$ has $k+1$ elements, we conclude that $m_i\le k$. As the identity~\eqref{identity} is taken from $J_k$, the word $x_{i1}\cdots x_{im_i}$ being in view of~\eqref{decomposition} a subword of length $\le k$ of the word $u_i$ must be a subword of a word in $\{v_1,\dots,v_r\}$. Thus, there is $r_i\in\{1,\dots,r\}$ such that
\[
v_{r_i}=v_{i1}x_{i1}v_{i2}x_{i2}\cdots x_{im_i}v_{i,m_i+1}
\]
for some words $v_{i1},\dots,v_{i,m_i+1}\in\Sigma^\ast$. Using the fact that the transformations in $\mathcal{S}$ are extensive and order preserving, we readily obtain that
\[
q_0.v_{r_i}\varphi\ge q_0.(x_1x_2\cdots x_m)\varphi=q_0.u_i\varphi,
\]
$i=i_1,\dots,i_p$. Since $\mathcal{S}$ is a join-semiring of order preserving and extensive transformations of $\langle Q;\le\rangle$, it follows that 
\[
\begin{aligned}
q_0.(u_1+\cdots+u_\ell)\varphi&{}=q_0.(u_{i_1}+\cdots+u_{i_p})\varphi\\
&{}\le q_0.(v_{r_{i_1}}+\cdots+v_{r_{i_p}})\varphi\\
&{}\le q_0.(v_1+\cdots+v_r)\varphi
\end{aligned}
\]
as required.
\end{proof}

\begin{corollary}
\label{reverse inclusion2}
$J_k\subseteq \mathsf{Id}(\mathcal{R}_{k+1})$.
\end{corollary}

\begin{proof}{}
Let $Q=\mathcal{B}^{k+1} \setminus\{(0,\dots,0)\}$ be the set of all non-zero $(k+1)$-vectors over the boolean semiring $\mathcal{B}=\langle\{0,1\};+,\cdot\rangle$. We equip the set $Q$ with the component-wise order $\le$ induced by the standard order $0<1$ in $\mathcal{B}$. Then $\langle Q,\le\rangle$ becomes a join-semilattice in which the longest chain has length $k+1$. The semigroup $\langle R_{k+1};\cdot\rangle$ acts on the set $Q$ by the usual matrix multiplication on the right: if $q=(q_i)\in Q$ and $\alpha=(\alpha_{ij})\in R_{k+1}$, then 
\[
q.\alpha=\left(\sum_{i=1}^{k+1}q_i\alpha_{i1},\dots,\sum_{i=1}^{k+1}q_i\alpha_{i\,k+1}\right).
\]
As it is noted in~\cite{Volkov-04}, this is a faithful representation of the semigroup by order preserving and extensive transformations of $\langle Q,\le\rangle$. Further, for any $q=(q_i)\in Q$ and $\alpha=(\alpha_{ij}),\beta=(\beta_{ij})\in R_{k+1}$ we have:
\[
\begin{aligned}
&q.(\alpha+\beta)=\left(\sum_{i=1}^{k+1}q_i(\alpha_{i1}+\beta_{i1}), \dots,\sum_{i=1}^{k+1}q_i(\alpha_{i\,k+1}+\beta_{i\,k+1})\right)\\
&=\left(\sum_{i=1}^{k+1}q_i\alpha_{i1}+\sum_{i=1}^{k+1}q_i\beta_{i1}, \dots,\sum_{i=1}^{k+1}q_i\alpha_{i\,k+1}+\sum_{i=1}^{k+1}q_i\beta_{i\,k+1}\right)\\
&=\left(\mathsf{max}\!\left(\sum_{i=1}^{k+1}q_i\alpha_{i1},\sum_{i=1}^{k+1}q_i\beta_{i1}\right), \dots,\mathsf{max}\left(\sum_{i=1}^{k+1}q_i\alpha_{i\,k+1},\sum_{i=1}^{k+1}q_i\beta_{i\,k+1}\right)\right)\\
&=\mathsf{sup}(q.\alpha,q.\beta).
\end{aligned}
\]
Now Proposition~\ref{order} applies.
\end{proof}

\begin{proof}{ of Theorem~\ref{finite basis}}
The ai-semirings $\mathcal{U}_1$, $\mathcal{R}_1$ and $\mathcal{C}_1$ are trivial and thus admit a finite basis of identities.
Denote by $\mathcal S_{k+1}=\langle S_{k+1};+,\cdot\rangle$ the subsemiring of $\mathcal U_{k+1}$ consisting of all \textit{stair triangular} matrices, i.e. matrices $\bigl(\alpha_{ij}\bigr)$ satisfying: if $\alpha_{ij} = 1$, $i < j$, then
\[
\alpha_{ii}=\alpha_{ii+1}=\cdots=\alpha_{ij}=\alpha_{i+1j}=\cdots=\alpha_{jj}=1.
\]
It is noticed in~\cite[Section~5]{Klima-Polak-10}, the monoid $\langle S_{k+1};\cdot\rangle$ is isomorphic to the monoid $\langle C_{k+1};\cdot\rangle$.
In fact, it is easy to see that the ai-semiring $\mathcal S_{k+1}$ is isomorphic to $\mathcal C_{k+1}$.
Further, it is shown in~\cite[Sections~4.1 and~5]{Klima-Polak-10} that $\mathsf{Id}(\mathcal{S}_{k+1})=J_k$ and the ai-semiring $\mathcal{S}_{k+1}$ is finitely based by the identity
\[
x_1\cdots x_{k+1}\approx \sum_{i=1}^{k+1} x_1\cdots x_{i-1}x_{i+1}\cdots x_{k+1},
\]
Since $J_k=\mathsf{Id}(\mathcal{C}_{k+1})=\mathsf{Id}(\mathcal{S}_{k+1})\supseteq \mathsf{Id}(\mathcal{U}_{k+1})\supseteq \mathsf{Id}(\mathcal{R}_{k+1})$, these facts and Corollary~\ref{reverse inclusion2} imply that the ai-semirings $\mathcal{U}_{k+1}$,  $\mathcal{R}_{k+1}$ and $\mathcal{C}_{k+1}$ satisfy the same identities and these identities admit a finite basis. Theorem~\ref{finite basis} is thus proved.
\end{proof}

\end{document}